\documentclass[11pt]{amsart}

\usepackage{amsmath, amssymb,  mathabx, amsfonts,enumerate}
\usepackage[all]{xy}
\usepackage{amscd,stmaryrd}
\usepackage{comment}
\usepackage{mathtools}
\usepackage{tikz} 
\usepackage{tikz-cd} 
\tikzcdset{diagrams={nodes={inner sep=7.5pt}}}

\usepackage{url}
\usepackage{hyperref}

\newtheorem{theorem}{Theorem}[section]
\newtheorem{lemma}[theorem]{Lemma}
\newtheorem{corollary}[theorem]{Corollary}
\newtheorem{proposition}[theorem]{Proposition}
\newtheorem{theoremletter}{Theorem}

 \theoremstyle{definition}
 \newtheorem{definition}[theorem]{Definition}

  \newtheorem*{example*}{Example}

\numberwithin{equation}{section}
 
\newcommand {\N}{\mathbb{N}} 
\newcommand {\Z}{\mathbb{Z}}

\newcommand {\C}{\mathbb{C}}

\newcommand{\LL}{\mathcal{L}}

\newcommand{\PP}{\mathcal{P}}

\DeclareMathOperator{\Ker}{Ker}

\DeclareMathOperator{\M}{Mat}

\DeclareMathOperator{\Id}{Id}

\DeclareMathOperator{\supp}{supp}
\DeclareMathOperator{\Spec}{Spec}

\begin{document}
\title[Twisted group rings and noisy linear cellular automata]{Stable finiteness of twisted group rings and noisy linear cellular automata}    
\author[Xuan Kien Phung]{Xuan Kien Phung}
\address{Département d'informatique et de recherche opérationnelle,  Université de Montréal, Montréal, Québec, H3T 1J4, Canada.}
\email{phungxuankien1@gmail.com}   

\begin{abstract}
For linear non-uniform cellular automata (NUCA) which are local perturbations of linear CA over a group universe $G$ and a finite-dimensional vector space alphabet $V$ over an arbitrary field $k$, we investigate their Dedekind finiteness property, also known as the direct finiteness property,  i.e., left  or right invertibility implies invertibility. We say that the  group $G$ is $L^1$-surjunctive, resp. finitely $L^1$-surjunctive, if all such linear NUCA are automatically surjective whenever they are stably injective, resp. when in addition $k$ is finite.  In parallel, we introduce the ring $D^1(k[G])$ which is the Cartersian product $k[G] \times (k[G])[G]$ as an additive group but the multiplication is twisted in the second component. The ring $D^1(k[G])$ contains naturally the group ring $k[G]$ and we obtain a dynamical characterization of its stable  finiteness for every field $k$ in terms of the finite $L^1$-surjunctivity of the group $G$, which holds for example when $G$ is residually finite or initially subamenable. Our results extend known results in the case of CA.  
\end{abstract}
\date{\today}
\maketitle
  
\setcounter{tocdepth}{1}

\section{Introduction}
In this paper, we investigate and establish  the relation between some extensions of two well-known conjectures in symbolic dynamics and ring theory, namely, Gottschalk's surjunctivity conjecture and Kaplansky's stable finiteness conjecture. More specifically, given a group $G$, a field $k$, and a finite set $A$,   
Kaplansky 
conjectured \cite{kap} that the group ring $k[G]$ is stably finite, i.e., 
every one-sided invertible element of the ring of square matrices of size $n \times n$  with coefficients in $k[G]$ must be a two-sided unit, while Gottschalk's surjunctivity conjecture \cite{gottschalk}  states that every injective $G$-equivariant uniformly continuous self-map $A^G \righttoleftarrow$ must be surjective. 
It is known that every one-sided unit of $\C[G]$ must be a two-sided unit \cite{kap}. 
Moreover, both conjectures are known for the  wide class of sofic groups  introduced by Gromov (see \cite{gromov-esav},  \cite{weiss-sgds}, \cite{elek}, \cite{ara}, \cite{csc-artinian}, \cite{li-liang}, \cite{phung-geometric}) but they are still open in general. As an application of our main results, we obtain an extension of the known  equivalence (cf. \cite[Theorem~B]{phung-weakly}, \cite[Theorem~B]{phung-geometric}) between Kaplansky's stable finiteness and a weak form of Gottschalk's surjunctivity conjecture. More precisely, 
we establish the equivalence between the surjunctivity property of locally disturbed  linear cellular automata (CA)  and 
the stable finiteness of some twisted group rings (Theorem~\ref{t:main-first-intro}).  
\par 
To state the main results, let us recall some  notions of symbolic dynamics. 
Given a discrete set $A$ and a group $G$, a \emph{configuration} $c \in A^G$ is a map $c \colon G \to A$.  Two configurations $x,y  \in A^G$ are \emph{asymptotic} if  $x\vert_{G \setminus E}=y\vert_{G \setminus E}$ for some finite subset $E \subset G$.  
The \emph{Bernoulli shift} action $G \times A^G \to A^G$ is defined by $(g,x) \mapsto g x$, 
where $(gx)(h) =  x(g^{-1}h)$ for  $g,h \in G$,  $x \in A^G$. 
We equip the \emph{full shift} $A^G$ with the \emph{prodiscrete topology}. 
For $x \in A^G$, we define $\Sigma(x) = \overline{\{gx \colon g \in G\}} \subset A^G$ as the smallest subshift containing  $x$. 
Following the idea of von Neumann and Ulam  \cite{neumann}, a CA over a group $G$ (the \emph{universe}) and a set $A$ (the \emph{alphabet})  is  a self-map $A^G \righttoleftarrow$ which is $G$-equivariant and uniformly continuous (cf.~\cite{hedlund-csc}, \cite{hedlund}). 
When different cells can evolve according to different local transition maps, we obtain  \emph{non-uniform CA} (NUCA). More precisely, we have (cf. \cite[Definition~1.1]{phung-tcs}, \cite{Den-12a}, \cite{Den-12b}):   

\begin{definition}
\label{d:most-general-def-asyn-ca}
Let $G$ be a group and let  $A$ be a set. Let $M \subset G$ be a subset and let $S = A^{A^M}$ be the set of all maps $A^M \to A$. Given $s \in S^G$, the NUCA $\sigma_s \colon A^G \to A^G$ is defined for all $x \in A^G$ and $g \in G$ by the formula 
\begin{equation*}
\sigma_s(x)(g)=  
    s(g)((g^{-1}x)  
	\vert_M). 
 \end{equation*}
 \end{definition} 
\par 
The set $M$ is called a \emph{memory} and $s \in S^G$  the \textit{configuration of local defining maps} of $\sigma_s$.  Every CA is thus a NUCA with finite memory and constant configuration of local defining maps. 
Following \cite{phung-tcs}, we say that  $\sigma_s$ is \emph{invertible} if it is bijective and the inverse map  $\sigma_s^{-1}$ is a NUCA with \emph{finite} memory. 
Moreover, $\sigma_s$ is \emph{left-invertible}, resp. \emph{right-invertible}, if $\tau \circ \sigma_s= \Id$, resp. $\sigma_s\circ \tau= \Id$, for some  NUCA $\tau \colon A^G \to A^G$ with \textit{finite} memory.  The NUCA $\sigma_s$ is \emph{pre-injective} if $\sigma_s(x) = \sigma_s(y)$ implies $x= y$ whenever $x, y \in A^G$ are asymptotic, and $\sigma_s$ is \emph{post-surjective} if for all $x, y \in A^G$ with $y$ asymptotic to $\sigma_s(x)$,  then   $y= \sigma_s(z)$ for some $z \in A^G$ asymptotic to $x$. We say that $\sigma_s$ is  \emph{stably injective} if  $\sigma_p$ is injective  for every $p \in \Sigma(s)$.  Similarly,  $\sigma_s$ is \emph{stably post-surjective} if $\sigma_p$ is post-surjective for every $p \in \Sigma(s)$. 
\par 
If $A$ is a vector space, $A^G$ is naturally a vector space with component-wise operation and we call a NUCA $\tau \colon A^G \to A^G$ \emph{linear} if it is also a linear map of vector spaces. Clearly, $\tau$ is a linear NUCA if and only if its local transition maps are all linear. Such linear NUCA with finite memory are interesting dynamical objects since they satisfy the shadowing property \cite{phung-shadowing}, \cite{phung-dual-linear-nuca}.   
\par 
\begin{definition}
    Given a group $G$ and a vector space $V$, we denote by $\mathrm{LNUCA}_{c}(G, V)$ the space of all linear NUCA $\tau \colon V^G \to V^G$ with finite memory which admit asymptotically constant configurations of local defining maps, i.e., $\tau \in \mathrm{LNUCA}_{c}(G, V)$ if there exist  finite subsets $M, E \subset G$ and $s \in \LL(V^M, V)^G$ such that $\tau=\sigma_s$ and $s(g)=s(h)$ for all $g,h \in G \setminus E$.  
\end{definition}
\par 
Let $G$ be a group and let $k$ be a field, it is not hard to deduce from \cite[Theorem~6.2]{phung-tcs} that $\mathrm{LNUCA}_{c}(G, k^n)$ is a $k$-algebra whose multiplication is given by the composition of maps and whose addition is component-wise. 
\par 
In parallel, we can define a generalization of the group ring $k[G]$, namely $D^1(k[G])$, which is given as the product  
$D^1(k[G])= k[G] \times (k[G])[G]$ with component-wise addition but where the multiplication is given by: 
\[
(\alpha_1, \beta_1) * (\alpha_2, \beta_2) = (\alpha_1 \alpha_2, \alpha_1 \beta_2 + \beta_1 \alpha_2 + \beta_1 \beta_2). 
\]
\par 
Here, the product $\alpha_1 \alpha_2$ is computed with the multiplication rule in the group ring $k[G]$ so that $k[G]$ is naturally a subring of $D^1(k[G])$ via the map $\alpha \mapsto (\alpha, 0)$. However,  $\alpha_1 \beta_2$,  $\beta_1 \alpha_2$, $\beta_1 \beta_2$ are  \emph{twisted products} (see  Definition~\ref{d:d-1-k-g}) which are  different from the products  computed with the multiplication rule of the group ring $(k[G])[G]$  with coefficients in $k[G]$. 
\par 
By \cite{csc-sofic-linear}, there exists a canonical ring isomorphism  between $M_n(k[G])$ and the ring $\mathrm{LCA}(G,k^n)$ of all linear CA $(k^n)^G \to (k^n)^G$.  
Extending the above isomorphism, we can also  interpret $\mathrm{LNUCA}_c(G,k^n)$ algebraically in terms of the ring $M_n(D^1(k[G]))$ as follows (see Theorem~\ref{t:iso-ring-first} and Proposition~\ref{p:iso-ring-second}): 
\begin{theoremletter}
\label{t:intro-iso-ring-first}
For every field $k$ and every infinite group $G$, there exists a canonical isomorphism $ \mathrm{LNUCA}_{c}(G, k^n)\simeq M_n(D^1(k[G]))$ for every $n \geq 1$. 
\end{theoremletter}
\par 
In \cite{csc-sofic-linear} and \cite{phung-weakly} respectively, the authors study the  
\emph{$L$-surjunctivity} and the \emph{linear surjunctivity} of a group, namely,  a group $G$ is  $L$-surjunctive, resp. linearly  surjunctive if for every finite-dimensional vector space $V$, resp. finite vector space $V$,   every injective $\tau  \in \mathrm{LCA}(G,V)$ is also surjective. It was shown  that all sofic groups are $L$-surjunctive \cite{gromov-esav}, \cite{csc-sofic-linear}. Notably, we know from \cite{csc-sofic-linear} that a group $G$ is $L$-surjunctive if and only if $k[G]$ is stably finite for every field $k$. Moreover, results in \cite{phung-weakly} show that $L$-surjunctivity and linearly surjunctivity are equivalent notions. In this vein, we introduce the following various notions of surjunctivity in the case of  linear NUCA.  

\begin{definition}
\label{d:intro-l-1-surjunctive}
Let $G$ be a group. 
We  say that  $G$ is \emph{$L^1$-surjunctive}, resp. \emph{finitely $L^1$-surjunctive},  if for every finite-dimensional vector space $V$, resp. for every  finite vector space $V$,  every stably  injective $\tau  \in \mathrm{LNUCA}_c(G,V)$ is also surjective. 
\end{definition}
\par 
In the line of some recent results which establish the multifold interaction between symbolic dynamics, group theory, and ring theory such as \cite{bartholdi-kielak}, \cite{cscp-jpaa}, \cite{phung-israel}, \cite{phung-geometric}, \cite{phung-weakly}, etc. our main result is the following: 
\begin{theoremletter}
\label{t:main-first-intro}
For every infinite group $G$, the following are equivalent:  
\begin{enumerate}[\rm (i)]
\item 
 $G$ is $L^1$-surjunctive; 
\item 
 $G$ is finitely $L^1$-surjunctive;
\item for every field $k$, the ring $D^1(k[G])$ is stably finite; 
\item 
for every finite field $k$, the ring $D^1(k[G])$ is stably finite; 
\item 
$G$ is dual $L^1$-surjunctive; 
\item 
$G$ is finitely dual $L^1$-surjunctive. 
\end{enumerate}
\end{theoremletter}
\par 
Here, a group $G$ is \emph{dual $L^1$-surjunctive}, resp. \emph{finitely dual  $L^1$-surjunctive},  if for every finite-dimensional vector space $V$, resp. for every  finite vector space $V$,  every stably post-surjective $\tau  \in \mathrm{LNUCA}_c(G,V)$ is pre-injective. 
\par 
The dual surjunctivity is studied in \cite{kari-post-surjective} where it was shown that every post-surjective CA over a sofic universe and a finite alphabet is also pre-injective. See also \cite{phung-post-surjective} for some extensions.  
As an application of Theorem~\ref{t:main-first-intro}, we obtain the following result which extends \cite[Theorem~B]{phung-tcs} and \cite[Theorem~D]{phung-dual-linear-nuca} to cover the case of initially subamenable group universes (see Section~\ref{s:sofic-groups}) and arbitrary finite-dimensional vector space alphabets:   
\begin{theoremletter}
    \label{t:intro-initially-subamenable} 
    All initially subamenable groups  and 
    all residually finite groups are $L^1$-surjunctive and dual $L^1$-surjunctive.   
\end{theoremletter}
\par 
We deduce immediately from Theorem~\ref{t:main-first-intro} and Theorem~\ref{t:intro-initially-subamenable} the following result on the stable finiteness of the twisted group rings:  
\begin{corollary}
    Let $G$  be a residually finite group or an initially subamenable group. Then for every field $k$, the ring $D^1(k[G])$ is stably finite. \qed 
\end{corollary}

\par 
The paper is organized as follows. We recall   in Section~\ref{s:sofic-groups} the definition of intially subamenable groups and residually finite groups. Section~\ref{s:induced-local-map} collects the construction of various induced local maps of NUCA. Then we establish the equivalence of the left-invertibility and the stable injectivity of elements of the class $\mathrm{LNUCA}_c(G,V)$ where $V$ is any finite-dimensional vector space (Theorem~\ref{t:general-stably-injective-left-invertible-linear}, Theorem~\ref{t:converse-general-stably-injective-left-invertible-linear}). The construction of the twisted group ring $D^1(k[G])$ is given in Section~\ref{s:4}. We then present the proof of Theorem~\ref{t:intro-iso-ring-first} as a consequence of Theorem~\ref{t:iso-ring-first} and Proposition~\ref{p:iso-ring-second} respectively in Section~\ref{s:5} and Section~\ref{s:6}. The dynamical characterization of the direct finiteness of the ring $M_n(D^1(k[G]))$ in terms of the direct finitness of $\mathrm{LNUCA}_c(G,k^n)$. 
The proof of the main result Theorem~\ref{t:main-first-intro} is contained in Section~\ref{s:main-general}. Finally, in Section~\ref{s:9}, we prove Theorem~\ref{t:intro-initially-subamenable} as an application of Theorem~\ref{t:main-first-intro}.

\section{Initially subamenable groups and residually finite groups}
\label{s:sofic-groups}

\subsection{Amenable groups} Amenable groups were defined by von Neumann \cite{neumann-amenable}.  A group $G$ is \emph{amenable} if  the Følner's condition \cite{folner} is satisfied: 
for every $\varepsilon >0$ and $T \subset G$ finite, there exists $F \subset G$ finite such that $\vert TF \vert \leq (1+\varepsilon)\vert F \vert$. Finitely generated groups of subexponential growth and  solvable groups are amenable. However, all groups containing a subgroup isomorphic to a free group of rank 2 are non-amenable. See e.g. \cite{stan-amenable} for some more details. 
The celebrated Moore and Myhill Garden of Eden theorem  \cite{moore}, \cite{myhill} was generalized to characterize amenable groups (cf.~\cite{ceccherini}, \cite{bartholdi-kielak}, \cite{cscp-alg-goe}, \cite{phung-2020},  \cite{phung-post-surjective}) and asserts that a CA with finite alphabet over an amenable group universe is surjective if and only if it is pre-injective. 
\par 
More generally, we say that  a group $G$ is \emph{initially subamenable} if for every
  $E \subset G$ finite,  there exist an amenable group $H$ and an injective map $\varphi \colon E \to  H$  such that $\varphi (gh)= \varphi(g) \varphi(h)$ for all $g,h \in E$ with $gh \in E$. Initially subamenable groups are sofic but the converse does not hold \cite{cornulier-example}.  
Note also that 
finitely presented initially subamenable groups 
are residually amenable but there exist initially subamenable groups which are not residually amenable \cite{elek-example}.

\subsection{Residually finite groups} 
We say that a group $G$ is \emph{residually finite} if for every finite subset $F \subset G$, there exists a finite group $H$ and a surjective group homomorphism $\varphi \colon G \to H$ such that the restriction   $\varphi\vert_F \colon F \to H$ is injective. All finitely generated abelian groups and 
more generally all finitely generated linear are   residually finite. Note that  free groups are non-amenable but residually finite.

\section{Induced local maps of NUCA} 
\label{s:induced-local-map}
To fix the notation, for all sets $E \subset F$ and $\Lambda \subset A^F$, we denote $\Lambda_E=\{ x\vert_{E}\colon x \in \Lambda \} \subset A^E$. 
Let $G$ be a group and let $A$ be a set. For every subset $E\subset G$, $g \in G$,  and $x \in A^E$ we define $gx \in A^{gE}$ by setting $gx(gh)=x(h)$ for all $h \in E$. In particular,  $gA^E= \{gx \colon x \in A^E\}=A^{gE}$. 
\par 
Let $M$  be a subset of a group $G$. Let $A$ be a set and let $S=A^{A^M}$. For every finite subset $E \subset G$  and $w \in S^{E}$,  
we define a map  $f_{E,w}^+ \colon A^{E M} \to A^{E}$ for every $x \in A^{EM}$ and $g \in E$ by (see e.g. \cite[Lemma~3.2]{cscp-alg-goe}, \cite[Proposition~3.5]{phung-2020}, \cite[Secion~2.2]{phung-embedding} for the case of CA):  
\begin{align}
\label{e:induced-local-maps} 
    f_{E,w}^+(x)(g) & = w(g)((g^{-1}x)\vert_M). 
\end{align}
\par 
In the above formula, note that  $g^{-1}x \in A^{g^{-1}EM}$ and $M \subset g^{-1}EM$ since $1_G \in g^{-1}E$ for $g \in E$. Therefore, the map   $f_{E,w}^+ \colon A^{E M} \to A^{E}$ is well defined. 
\par
Consequently, for every $s \in S^G$, we have a well-defined induced local map $f_{E, s\vert_E}^+ \colon A^{E M} \to A^{E}$ for every finite subset $E \subset G$ which satisfies: 
\begin{equation}
\label{e:induced-local-maps-general} 
    \sigma_s(x)(g) =  f_{E, s\vert_E}^+(x\vert_{EM})(g)
\end{equation}
for every $x \in A^G$ and $g \in E$. Equivalently, we have for all $x \in A^G$ that: 
\begin{equation}
\label{e:induced-local-maps-proof} 
    \sigma_s(x)\vert_E =  f_{E, s\vert_E}^+(x\vert_{EM}). 
\end{equation}

\section{Left-invertibility of stably injective linear NUCA} 
\label{s:4}

For the proof of the main result of the section Theorem~\ref{t:general-stably-injective-left-invertible-linear}, 
we shall need the following useful technical lemma.  

\begin{lemma}
\label{l:reversible-finite-memory-linear}
Let $G$ be a finitely generated infinite group and let $V$ be a finite-dimensional vector space. 
Let $\tau \in \mathrm{LNUCA}_c(G,V)$ be a stably  injective linear NUCA and 
let $\Gamma\coloneqq \tau(V^G)$. 
Then there exists a finite subset $N\subset G$ such that the following condition holds: 
\begin{enumerate}
\item[$\mathrm{(C)}$]  
 \textit{for any} $d\in \Gamma$ \textit{and} $g \in G$, the element  
 $\tau^{-1}(d)(g)\in V$ \textit{depends only on the restriction} $d \vert_{gN}$. 
\end{enumerate}
\end{lemma}

\begin{proof}
Since $\tau$ is a linear NUCA with finite memory, there exists a finite subset $M \subset G$ and $s \in S^G$ where $S= \LL(V^M, V)$ such that $\tau= \sigma_s$. By hypothesis, $s$ is asymptotic to a constant configuration $c \in S^G$. Up to enlarging $M$, we can also suppose that $s\vert_{G \setminus M}= c\vert_{G \setminus M}$ and that $1_G \in M$. 
Since the group $G$ is finitely generated thus countable, it admits an increasing sequence of finite subsets 
$M=E_0 \subset \dots \subset E_n \dots$ such that $G=\cup_{n\in \N} E_n$. 
\par
Suppose on the contrary that there does not exist a finite subset $N\subset G$ which satisfies condition  $\mathrm{(C)}$. 
Then by linearity, there exist  for each $n \in \N$ a configuration $d_n \in \Gamma$ and an element $g_n \in G$ such that for $c_n=\tau^{-1}(d_n)$ (which is well-defined since $\tau$ is injective), we have: 
\begin{equation*}
d_n\vert_{g_nE_n}=0^{g_nE_n} \quad \text{and} \quad c_n(g_n)\neq 0.  
\end{equation*}
 \par 
Consequently, by letting $x_n=g_n^{-1}c_n$ and $y_n= g_n^{-1}d_n$, we infer from \cite[Lemma~5.1]{phung-tcs} that $\sigma_{g_n^{-1}s}(x_n)=y_n$ and: 
\begin{equation*}
y_n\vert_{E_n}=0^{E_n} \quad \text{and} \quad x_n(1_G)\neq 0.  
\end{equation*}
\par 
Since $s$ is asymptotic to a constant configuration $c \in S^G$ by hypothesis, the set $T=\{s(g) \colon g \in G\}$ is actually a finite subset of $S= \LL(V^M, V)$. It follows that $\Sigma(s) \subset T^G$ is a compact subspace. Therefore, up to restricting to a subsequence, we cans suppose without loss of generality that the sequence $(g_n^{-1}s)_{n \in \N}$ converges to a configuration $t \in T^G \subset S^G$ with respect to the prodiscrete topology. 
 \par 
By \cite[Lemma~8.1]{phung-tcs}, we know that $\Sigma(s) = \{gs \colon g \in G\} \cup \{c\}$. Note that if $s$ is constant then the lemma results from \cite{csc-sofic-linear}. Hence, we can suppose in the sequel that $s$ is not a constant configuration. In particular, $T$ and $\Sigma(s)$ are not singleton. We distinguish two cases according to whether $t = c$ or not. 
\par 
\textbf{Case 1:} $t=gs$ for some $g \in G$. Then since $G$ is infinite and $s$ is asymptotic but not equal to $c$, we can, up to restricting to a subsequence again, assume without loss of generality that $g_n^{-1}=g$ for all $n \in \N$. Up to replacing $s$ by $gs$, we can also suppose that $g=1_G$ so that $\sigma_s(x_n)=y_n$ for all $n \in \N$. 
For each $n\in \N$, consider the following linear subspace of $V^{E_nM}$: 
\[
I_n\coloneqq  \Ker(f^+_{E_n, s\vert_{E_n}}) \subset V^{E_nM}.
\]
\par 
Observe that $x_n\vert_{E_nM} \in I_n \setminus \{0^{E_nM} \}$. 
Note also that for all $n \leq m \leq k$, the projection 
$p_{nm} \colon A^{E_m} \to A^{E_n}$ induces a linear map 
$\pi_{nm} \colon I_m \to I_n$  and $\pi_{nk}(I_{k}) \subset \pi_{nm}(I_m)$. Hence, for each $n\in \N$, we obtain a decreasing sequence 
of linear subspaces $(\pi_{nm}(I_m))_{m \geq n}$ of $I_n$. 
Hence, $(\pi_{nm}(I_m))_{m \geq n}$ is stationary and 
there exists a linear subspace $J_n \subset I_n\subset A^{E_nM}$ 
such that $\pi_{nm}(I_m)=J_n$ for all $m$ large enough. 
\par
Observe that  $\pi_{nm}(J_m) \subset J_n$ for all $m \geq n$. 
We claim that the restriction  linear map 
$q_{nm} \colon J_m  \to J_n$  is surjective for all $m\geq n$. 
Indeed, let $y\in J_n$ and let $k \geq m$ be sufficiently large  
such that $q_{n k}(I_k)=J_n$ and $q_{m k}(I_k)=J_m$. 
Thus, $q_{nk}(x)=y$ for some  $x\in I_k$. 
As $q_{nk}= q_{nm} \circ q_{mk}$, we have $q_{nm}(y')=y$ where $y'=q_{mk}(x) \in J_m$. The claim is proved.
\par 
We choose $k\in \N$ large enough  such that $\pi_{0k}(I_k)=J_0$. Let $z_0=\pi_{0k}(x_k) \in J_0$ then $z_0(1_G) \neq 0$. 
We define by induction a sequence $(z_n)_{n \in \N}$ where 
$z_n \in J_n$ for all $n\in \N$ as follows. 
Given $z_n \in J_n$ for some $n\in \N$, there exists 
by the surjectivity of the map $q_{n,n+1}$ an element 
\[
z_{n+1}\in q^{-1}_{n,n+1}(z_n) \subset J_{n+1}\subset A^{E_{n+1}M}.
\] 
\par 
We thus obtain a configuration $c \in V^G$ defined by $z\vert_{E_nM}=z_n$ for all $n \in \N$. Since $G=\cup_{n \in \N} E_n M$, the configuration $z$ is well-defined.  
\par
By construction, we have  for all $n \in \N$ that: 
\[
\tau(z)\vert_{E_n}=f^+{E_n, s\vert_{E_n}}(z\vert_{E_nM})=f^+{E_n, s\vert_{E_n}}(z_n)=0^{E_n}. 
\]
\par 
Therefore, $\tau(z)= 0^{G}$ but $z(1_G)\neq 0$ which then  contradicts the injectivity of the linear NUCA $\tau$. 
\par 
\textbf{Case 2:}   $t = c$. Then since $\lim_{n \to \infty} g_n^{-1}s=t$ and $s \neq c$, we deduce immediately that $g_n$ escapes finite groups when $n \to \infty$, i.e., for every finite subset $E \subset G$, there exists $N \in \N$ such that $g_n \notin E$ for all $n \geq N$. Consequently, by restricting to a suitable subsequence, we can suppose without loss of generality that $g_nE_nM \cap M = \varnothing$ for all $n \in \N$. As $s\vert_{G \setminus M}= c\vert_{G \setminus M}$, it follows that $(g_n^{-1}s)\vert_{E_nM} = c\vert_{E_nM}$ for all $ n \in \N$. 
Since $c$ is constant, we deduce that $\sigma_c(x_n)\vert_{E_n}=0^{E_n}$  and $x_n(1_G)\neq 0$. We infer from the stable injectivity of $\sigma_s$ that $\sigma_c$ is injective. Therefore, a similar argument as in \textbf{Case 1} applied for $\sigma_c$ and the sequence $(x_n)_{n \in \N}$ leads to a contradiction. 
\par 
Consequently, there must exist a finite subset $N \subset G$ which satisfies condition (C) and the proof is thus complete. 
\end{proof}

\par 
Our next results Theorem~\ref{t:general-stably-injective-left-invertible-linear} and Theorem~\ref{t:converse-general-stably-injective-left-invertible-linear} extend the results \cite[Theorem~10.1]{phung-dual-linear-nuca} and \cite[Theorem~7.1]{phung-tcs} for NUCA over finite alphabet to the class $\mathrm{LNUCA}_c$ over an arbitrary finite-dimensional vector space. 

\begin{theorem}
\label{t:general-stably-injective-left-invertible-linear} 
Let $G$ be a  group and let $V$ be a finite-dimensional vector space. 
Let $\tau\in \mathrm{LNUCA}_c(G,V)$ be a stably  injective linear NUCA.  Then $\tau$ is left-invertible, i.e., there exists  $\sigma \in \mathrm{LNUCA}_c(G,V)$  such that $\sigma \circ \tau= \Id$.  
\end{theorem}

\begin{proof}
As the linear NUCA $\tau$ has  finite memory, we can find a finite subset $M \subset G$ and $s \in S^G$ where $S= \LL(V^M, V)$ such that $\tau= \sigma_s$. By hypothesis, the configuration $s$ is asymptotic to a constant configuration $c \in S^G$. Hence,  we can, up to enlarging $M$, suppose that $s\vert_{G \setminus M}= c\vert_{G \setminus M}$ and that $1_G \in M$. 
\par 
Assume first that $G$ is a finitely generated infinite group. Then we infer from 
Lemma~\ref{l:reversible-finite-memory-linear} that there exists a finite subset $N \subset G$ such that for any $d\in \tau(V^G)$  and $g \in G$, the element  
 $\tau^{-1}(d)(g)\in V$  depends only on the restriction $d \vert_{gN}$. Up to enlarging $M$ and $N$, we can clearly suppose that $M=N$. Consequently, for each $g \in G$, we have a well-defined map $\varphi_g\colon \tau(V^G)_{gM} \to V$ given by $d\vert_{gM} \mapsto \tau^{-1}(d)(g)$ for every $d \in V^G$. \par 
 Since $\tau$ is linear and $\tau(V^G)_{gM}$ is a linear subspace of $V^{gM}$, it follows that $\varphi_g$ is also a linear map and we can extend $\varphi_g$ to a linear map $\tilde{\varphi}_g \colon V^{gM} \to V$ which coincides with $\varphi_g$ on $\tau(V^G)_{gM}$.  
 Let $\phi_g \colon V^M \to V^{gM}$ be the canonical automorphism induced by the bijection $M \simeq gM$, $h \mapsto gh$. 
 Let us define an configuration $t \in S^G$ where $S= \LL(V^M, V)$ by setting  $t(g)= \tilde{\varphi}_g \circ \phi_g \colon V^M \to M$ for every $g \in G$. 
 It is immediate from the construction that for every $c \in V^G$, $g \in G$, and $d = \tau(c) \in \tau(V^G)$, we have 
 \begin{equation*}
     \sigma_t(\sigma_s(c))(g)= \sigma_t(d)(g)=t(g)((g^{-1}d)\vert_{M})= \tau^{-1}(d)(g)=c(g). 
 \end{equation*}
 \par 
 Therefore, $\sigma_t \circ \sigma_s = \Id$ and we conclude that $\tau=\sigma_s$ is left-invertible. In fact, since $s\vert_{G \setminus M}= c\vert_{G \setminus M}$, the linear spaces $W=\phi_g^{-1}(\tau(V^G)_{gM}) =\phi^{-1}( f^+{gM,s\vert_{gM}}(V^{gM^2}))$ coincide as linear subspaces of $V^M$ for all $g \in G \setminus  M M^{-1}$.  Let us fix  
 a direct sum decomposition $V^M = W \oplus U$ of $V^M$. Thus, if we define $\tilde{\varphi}_g$  by setting $\tilde{\varphi}_g(v)=0$ for all $v \in \phi_g(U)$ and $\tilde{\varphi}_g(v)=\varphi_g(v)$ if $v \in \phi_g(W)$ and extend by linearity on the whose space $V^{gM}$, then it is clear that $t$ is also asymptotically constant, which completes the proof of the theorem in the case when $G$ is a finitely generated infinite group.  
 \par 
The case when $G$ is a finite group is trivial since every injective endomorphism of a finite-dimensional vector space is an automorphism. Let us consider the general case where $G$ is an infinite group. Let $H$ be the subgroup of $G$ generated by $M$.  
Let  
$G/H=\{gH \colon g \in G\}$ be the set of all right cosets of $H$ in $G$. By identifying $x \in A^G$  with $(x\vert_{u})_{u \in G/H}$,  we obtain a factorization $A^G = \prod_{u \in G/H} A^u$. Moreover,  $\sigma_s= \prod_{u \in G/H} \sigma_s^u$ where $\sigma_s^u\colon A^u \to A^u$ is given by  $\sigma_{s}^u (y) = \sigma_s(x)\vert_u$ for all $y \in A^u$ and any $x \in A^G$ extending $y$. Similarly, we have $\sigma_c= \prod_{u \in G/H} \sigma_c^u$.
\par For every coset $u \in G/H$, let us choose $g_u\in G$ such that $g_H=1_G$. Then if $u \neq H$, we have $s\vert_u =c\vert_u$ and $\sigma_s^u=\sigma_c^u$ is conjugate to the restriction CA $ \sigma_{c\vert_H} = \sigma_{c}^H \colon A^H \to A^H$ by the uniform homeomorphism $\phi_u \colon A^u \to A^H$ given by  $\phi_u(y)(h)=y(g_uh)$ for all $y \in A^u$ and $h \in H$ (cf. the discussion  following \cite[Lemma~2.8]{cscp-invariant}). 
Hence, $\sigma_s$ and $\sigma_c$ are left-invertible (resp. injective) if and only if so are $\sigma_{s\vert_H}$ and $\sigma_{c \vert_H}$ (see also \cite[Theorem~1.2]{induction-restriction}). Consequently, the general case follows from the case when $G$ is finite or when $G$ is a finitely generated infinite group. The proof is thus complete. 
\end{proof}
\par 
Conversely, we show that left-invertibility implies stable injectivity for linear NUCA with finite memory whose configuration of local defining maps is asymptotically constant. 

\begin{theorem}
\label{t:converse-general-stably-injective-left-invertible-linear} 
Let $G$ be a  group and let $V$ be a finite-dimensional vector space. 
Suppose that $\tau \in \mathrm{LNUCA}_c(G,V)$ is a left-invertible  linear NUCA. Then $\tau$ stably  injective. 
\end{theorem}

\begin{proof}
As in the proof of Theorem~\ref{t:general-stably-injective-left-invertible-linear}, we can suppose without loss of generality that $G$ is a finitely generated infinite group. 
Since $\tau$ is a linear NUCA with finite memory and left-invertible, we can find a finite subset $M\subset G$ and $s,t \in S^G$ where $S= \LL(V^M, V)$ such that $\tau=\sigma_s$ and $\sigma_t \circ \sigma_s= \Id$. In particular, we deduce immediately that $\sigma_s$ is injective. 
\par 
As $s$ is asymptotically constant, we infer from \cite[Lemma~8.1]{phung-tcs} that $\Sigma(s)=\{g s \colon g \in G\}\cup \{c\}$ for some constant configuration $c \in S^G$. 
Note that by \cite[Lemma~5.1]{phung-tcs}, the injectivity of $\sigma_{gs}$ for all $g \in G$  follows from the injectivity of $\sigma_s$. 
We must show that $\sigma_c$ is injective. For this, we can suppose, up to enlarging $M$, that $s\vert_{G \setminus M}= c\vert_{G \setminus M}$. Since $G$ is infinite, there exists $g \in G$ such that $g M \cap M= \varnothing$. It follows that 
$s\vert_{gM}=c\vert_{gM}$. 
On the other hand, we infer from  the identity $\sigma_t \circ \sigma_s= \Id$ that 
\[
t(g) \circ f^+_{gM, s\vert_{gM}}= \pi_{gM^2, g}
\]
where $\pi_{F,E} \colon V^F \to V^E$ denotes the canonical projection induced by any inclusion of sets $E \subset F$. 
Consequently, $t(g) \circ f^+_{gM, c\vert_{gM}}= \pi_{gM^2, g}$. 
Since $c$ is constant, we deduce that 
$\sigma_d \circ \sigma_c= \Id$ where $d \in S^G$ is the constant configuration defined by $d(h)=t(g)$ for all $h \in G$. In particular, $\sigma_c$ is injective and we conclude that $\sigma_s$ is stably injective. The proof is thus complete. 
\end{proof}

\section{The twisted group ring $D^1(k[G])$} 
\label{s:5}
Given a group $G$ and a ring $R$ (with unit), recall that the group ring $R[G]$ is the $R$-algebra which admits  $G$ as a basis and whose  multiplication is defined by the group product on basis elements and the distributive law. 

\begin{definition}
\label{d:d-1-k-g}
   Let $k$ be a ring and let $G$ be a group. We define $D^1(k[G])$ as the Cartesian product 
   \[
D^1(k[G])= k[G] \times (k[G])[G]. 
   \]
\par Elements of $D^1(k[G])$ are couples $(\alpha, \beta)$ where $\alpha \in k[G]$ is called the \emph{regular part} and $\beta \in (k[G])[G]$ is called the \emph{singular part} of $(\alpha, \beta)$. The addition operation of $D^1(k[G])$ is component-wise:  
\[
(\alpha_1, \beta_1) + (\alpha_2, \beta_2)= (\alpha_1+\alpha_2, \beta_1 + \beta_2). 
\]
\par 
The multiplication operation  $*\colon D^1(k[G]) \times D^1(k[G]) \to D^1(k[G])$ is defined as follows: 
\[
(\alpha_1, \beta_1) * (\alpha_2, \beta_2) = (\alpha_1 \alpha_2, \alpha_1 \beta_2 + \beta_1 \alpha_2 + \beta_1 \beta_2).  
\]
\par 
Here, $\alpha_1 \alpha_2$ is computed with the multiplication rule in the group ring $k[G]$. However, for $\alpha \in k[G]$ and $\beta, \gamma \in (k[G])[G]$, we define, by abuse of notation, the \emph{twisted products} $\alpha \beta$, $\beta \alpha$, and $\beta \gamma$ as elements of $(k[G])[G]$ as follows, which should be distinguished from the multiplication rule of the group ring $(k[G])[G]$  with coefficients in $k[G]$. 
Let $g, h\in G$, we set:   
\begin{align}
\label{e:def-singular-multiplication} 
    (\alpha \beta)(g)(h)& = \sum_{t \in G} \alpha(t) \beta(gt)(t^{-1}h),
\\
    (\beta \alpha)(g)(h)&= \sum_{t \in G}  \beta(g)(t) \alpha(t^{-1}h), \nonumber
\\
    (\beta\gamma)(g)(h)&= \sum_{t \in G} \beta(g)(t) \gamma(gt)(t^{-1}h). \nonumber 
\end{align}
\end{definition}

\par 
It is not hard to check that the above product rule is associative and distributive with respect to addition.  For example, with the above $\alpha, \beta, \gamma$, and $g,h$, we have:

 \begin{align*}
   (\alpha (\beta\gamma))(g)(h) &= \sum_{t \in G}  \alpha(t) \sum_{q \in G} \beta(gt)(q) \gamma(gtq)(q^{-1}t^{-1}h)\\
   & = \sum_{t \in G}  \alpha(t) \sum_{r\in G} \beta(gt)(t^{-1}r) \gamma(gr)(r^{-1}h) \qquad (r=tq)\\
   & =  \sum_{r \in G} \sum_{t \in G}\alpha(t) \beta(gt)(t^{-1}r)\gamma(gr)(r^{-1}h) \\
   & =  ((\alpha \beta)\gamma)(g)(h). 
\end{align*}

The following   lemma tells us that $D^1(k[G])$ is indeed a ring with unit.  

\begin{lemma}
    For every group $G$ and every ring $k$, the set $D^1(k[G])$ equipped with the addition and multiplication operations as defined in Definition~\ref{d:d-1-k-g} is a ring with unit $(1_G, 0)$ and neutral element $(0,0)$. 
\end{lemma}

\begin{proof}
Since the addition is component-wise and $k[G]$ and $(k[G])[G]$  are abelian groups, $D^1(k[G])$ is also an abelian group.  
It is clear that $(\alpha, \beta)*(1_G, 0)=(1_G, 0)*(\alpha, \beta)=(\alpha, \beta)$ 
for all $(\alpha, \beta) \in D^1(k[G])$. Moreover, the associativity of the multiplication is satisfied since  for all $(\alpha_i, \beta_i) \in D^1(k[G])$ ($i=1,2,3$), we find that: 
\begin{align*} 
& \left( (\alpha_1, \beta_1) * (\alpha_2, \beta_2)\right)  * (\alpha_3, \beta_3)  = (\alpha_1 \alpha_2, \alpha_1 \beta_2 + \beta_1 \alpha_2 + \beta_1 \beta_2) * (\alpha_3, \beta_3) \\
& = 
(\alpha_1 \alpha_2 \alpha_3, \alpha_1 \alpha_2 \beta_3+ \alpha_1 \beta_2 \alpha_3 + \beta_1 \alpha_2  \alpha_3+ \beta_1 \beta_2  \alpha_3+  
\alpha_1 \beta_2 \beta_3+ \beta_1 \alpha_2\beta_3 + \beta_1 \beta_2\beta_3)
\\
&=
(\alpha_1 \alpha_2 \alpha_3, \alpha_1 \alpha_2 \beta_3+ \alpha_1 \beta_2 \alpha_3 + 
\alpha_1 \beta_2 \beta_3+ \beta_1 \alpha_2 \alpha_3+ \beta_1 \alpha_2\beta_3 + \beta_1 \beta_2  \alpha_3+  
 \beta_1 \beta_2\beta_3) 
\\
& =  (\alpha_1, \beta_1) * (\alpha_2 \alpha_3, \alpha_2 \beta_3 + \beta_2 \alpha_3 + \beta_2 \beta_3)
\\
& = (\alpha_1, \beta_1) * \left( (\alpha_2, \beta_2) * (\alpha_3, \beta_3) \right). 
\end{align*} 
\par 
Finally, we see without difficulty that  the distributivity of $D^1(k[G])$ follows from the distributivity of $k[G]$ and $(k[G])[G]$. Hence, we conclude that $D^1(k[G])$ is a ring with  unit $(1_G, 0)$ and neutral element $(0,0)$. 
\end{proof} 
\par 
The next lemma says that the generalized group ring $D^1(k[G])$ contains naturally the group ring $k[G]$ as the subring of regular elements, i.e., elements whose singular parts are zero. 
\begin{lemma}
    \label{l:d-1-k-g-contains-k-g} 
Let $k$ be a ring and let $G$ be a group. Then we have a canonical  embedding of rings $\varphi \colon k[G] \hookrightarrow D^1(k[G])$ given by the formula $\varphi(\alpha)=  (\alpha,0)$ for all $\alpha \in k[G]$.
\end{lemma}

\begin{proof}
The map $\varphi$ is trivially injective. 
Moreover, it is a direct consequence of the definition of the addition of multiplication operations of $D^1(k[G])$ that $\varphi(\alpha_1) + \varphi(\alpha_2) = \varphi(\alpha_1+\alpha_2)$ and 
$\varphi(\alpha_1 \alpha_2) =\varphi(\alpha_1\alpha_2)$ for all $\alpha_1, \alpha_2 \in k[G]$. 
\end{proof}

\section{Non-uniform linear NUCA $(k^n)^G \righttoleftarrow$ and $D^1(M_n(k)[G])$}
\label{s:6}
Let $k$ be a field and let $G$ be a group. Let us fix an integer $n \geq 1$ and denote $V= k^n$. 
Recall that $\mathrm{LNUCA}_{c}(G, k^n)$ is the $k$-algebra of all linear NUCA with finite memory $\tau \colon (k^n)^G \to (k^n)^G$ which admit asymptotically constant  configurations of local defining maps. The multiplication of $\mathrm{LNUCA}_c(G,k^n)$ is given by the composition of maps and whose addition is component-wise. 
\par 
With every element $\omega = (\alpha, \beta) \in D^1(M_n(k)[G])$, we can associate a map $\tau^\omega \colon V^G \to V^G$ defined as follows: 
\begin{align}
\label{e:definition-tau-omega-linear}
\tau^\omega(x)(g) = \sum_{h \in G} \alpha(h) x(gh) + \sum_{h \in G} \beta(g)(h)x(gh)  \quad 
     \text{for all } x\in V^G, g \in G.  
\end{align}
\par 
For every element $\gamma \in M_n(k)[G]$, we denote the \emph{support} of $\gamma$ as the finite subset  $\supp(\gamma) = \{ g \in G\colon \gamma(g) \neq 0\}$ of $G$. 
\par 
\begin{lemma}
    \label{l:well-defined-tau-omega-linear-nuca} 
   The map $\tau^\omega \colon V^G \to V^G$ is a linear NUCA with finite memory. Moreover, $\tau^\omega$ admits a configuration of local defining maps which is asymptotic to a constant configuration, i.e., 
   $\tau^{\omega} \in \mathrm{LNUCA}_{c}(G, k^n)$. 
\end{lemma}

\begin{proof}
Since $\alpha(g) , \beta(g)(h) \in M_n(k)$ for all $g,h \in G$, it follows from \eqref{e:definition-tau-omega-linear} that $\tau^\omega$ is a linear map. Consider the following finite subset of $G$: 
\[
M =  \cup_{g \in G} \supp \beta(g) \cup \supp \alpha.
\]
\par 
We define a configuration of local defining maps $s \in S^G$ where $S= \LL(V^G, V)$ as follows. For every $g \in G$, let $s(g) \in S$ be the linear map determined for all $w \in V^M$ by: 
\[
s(g)(w) = \sum_{h \in M} \alpha(h)w +  \sum_{h \in M} \beta(g)(h)w.  
\]
\par 
Let $x \in V^G$ and $g \in G$. Then we infer from the definition \eqref{e:definition-tau-omega-linear} and the choice of $M$ that: 
\begin{align*}
    \sigma_s(x)(g) &= s(g)((g^{-1}x)\vert_M) \\ 
    & = \sum_{h \in M} \alpha(h)x(gh) +  \sum_{h \in M} \beta(g)(h)x(gh)\\
    & = \sum_{h \in G} \alpha(h)x(gh) +  \sum_{h \in G} \beta(g)(h)x(gh)\\
    & = \tau^\omega(x)(g). 
\end{align*}
\par 
We deduce that $\tau^\omega= \sigma_s$ is indeed a linear NUCA with finite memory. On the other hand, if we denote $E= \supp \beta$ then $E$ is a finite subset of $G$ and we have $s(g)= \alpha(g)$ for all $g \in G \setminus E$ by construction. Consequently, $s$ is asymptotic to the constant configuration $\alpha^G$. Thus, 
$\tau^\omega \in \mathrm{LNUCA}_c(G, k^n)$ and 
the proof is complete.    
\end{proof}
\par 
It turns out that the converse of the above lemma also holds. In other words, every linear NUCA over $V^G$ with finite memory and asymptotically constant configuration of local defining maps arises uniquely as a map $\tau^\omega$ described above.  
More specifically, the following results says that the map $\omega \mapsto \tau^{\omega}$ is a ring isomorphism when $G$ is infinite.  
\begin{theorem}
\label{t:iso-ring-first}
     Let $k$ be a field and let $G$ be an infinite group. Then for every integer $n \geq 1$, the map $\Psi \colon D^1(M_n(k)[G]) \to \mathrm{LNUCA}_{c}(G, k^n)$ given by $\omega \mapsto \tau^{\omega}$ is a $k$-linear ring isomorphism. 
\end{theorem}

\begin{proof}
Let $V= k^n$. We claim that $\Psi$ is injective. Indeed, let $\omega=(\alpha, \beta) \in D^1(M_n(k)[G])$ be an element such that $\tau^\omega = 0$ as a map from $V^G$ to itself. Let $M =  \cup_{g \in G} \supp \beta(g) \cup \supp \alpha$ then $M$ is a finite subset of $G$. Since $G$ is infinite, we can choose some $g_0 \in G \setminus M$. In particular, $\beta(g_0)=0$ by the choice of $M$. 
Then for every $x \in V^G$, we find that $\tau^\omega(x)(g_0)=0$ and  it follows from  \eqref{e:definition-tau-omega-linear} that 
\[
\sum_{h \in M} \alpha(h)x(g_0h) =
\sum_{h \in G} \alpha(h)x(g_0h)=\tau^\omega(x)(g_0)=0.
\]
\par 
Since $x$ is arbitrary, we deduce that $\alpha(h)$ for all $h \in M$ and thus $\alpha=0$. Consequently, we infer again from  \eqref{e:definition-tau-omega-linear} that for all $g \in G$: 
\[
\sum_{h \in M} \beta(g)(h)x(gh)=\sum_{h \in G} \beta(g)(h)x(gh)=0
\] 
\par  Thus, $\beta(g)(h)=0$ for all $g,h \in G$. In other words, $\beta=0$ and we conclude that $\omega=0$. Hence, $\Psi$ is indeed injective as claimed. 
\par 
To check that $\Psi$ is surjective, let $\sigma_s \in \mathrm{LNUCA}_c(G, V)$ where $s \in S^G$ for some $S=\LL(V^M, V)$,  where $M \subset G$ is a finite subset, such that $s$ is asymptotic to a constant configuration $c \in S^G$. Up to enlarging $M$, we can also suppose that $s\vert_{G \setminus M} = c\vert_{G \setminus M}$. 
\par 
Since $c(1_G) \in \LL(V^M, V)$, there exist $\gamma_h \in \LL(V,V)=\M_n(k)$ for every $h \in M$ such that for all $w \in V^M$, we have $c(1_G)(w) = \sum_{h \in M} \gamma_h w$. 
Let us denote $\alpha = \sum_{h \in M}\gamma_h h \in M_n(k)[G]$.
\par 
For each $g \in M$, we define   $ \delta_g=s(g) - c(g) \in \LL(V^M,V)$. By linearity, there exists uniquely $\delta_g(h) \in M_n(k)= \LL(V,V)$ for $h \in M$ such that for all $w \in V^M$, we have: 
\[
\delta_g(w) = \sum_{h \in M} \delta_g(h)w(h). 
\]
\par 
Hence, we obtain an element $\mu_g = \sum_{h \in M} \delta_g(h) \in M_n(k)[G]$ for every $g \in M$. 
Let us denote $\beta= \sum_{g \in M} \mu_g g \in (M_n(k)[G])[G]$ 
and $\omega=(\alpha, \beta) \in D^1(M_n(k)[G])$. We claim that $\tau^\omega = \sigma_s$. Indeed, 
for every $x \in V^G$ and $g \in G$, we find that: 

\begin{align*}
\tau^\omega(x)(g) & = \sum_{h \in G} \alpha(h) x(gh) + \sum_{h \in G} \beta(g)(h)x(gh) \\
& = \sum_{h \in M} \gamma_h x(gh) + \sum_{h \in M} \mu_g(h)x(gh)\\
&=  c(1_G)((g^{-1}x)\vert_M) + \sum_{h \in M} \delta_g(h)x(gh)\\
& = c(g)((g^{-1}x)\vert_M) + \delta_g((g^{-1}x)\vert_M)\\
& = s(g)((g^{-1}x)\vert_M)\\
&=\sigma_s(x)(g).  
\end{align*}
\par 
We conclude that $\sigma_s = \tau^\omega=\Psi(\omega)$ from which follows the surjectivity of the map  $\Psi$. It is immediate from \eqref{e:definition-tau-omega-linear} that $\Psi$ is a $k$-linear homomorphism of groups and $\Psi$ sends the unit of $D^1(M_n(k)[G])$ to the unit of $\mathrm{LNUCA}_c(G, V)$. 
\par 
To finish the proof, we have to check that for all elements  $\omega=(\alpha, \beta)$ and $\omega'=(\alpha', \beta')$ of $D^1(M_n(k)[G])$, we have $\Psi(\omega * \omega') = \Psi(\omega) \circ \Psi(\omega')$. Indeed, using the formula  
$\omega * \omega' = (\alpha \alpha', \alpha \beta' + \beta \alpha'+ \beta \beta')$ and \eqref{e:def-singular-multiplication}, 
we can compute for all $x \in V^G$,  $g \in G$, and $y = \tau^{\omega'}(x)$ that:   

\begin{align*}
& \Psi(\omega * \omega') (x)(g) 
=  \sum_{h \in G} \left( \alpha\alpha'  +   (\alpha \beta' + \beta \alpha'+ \beta \beta')(g)\right)(h) x(gh) \\
& = \sum_{h, t \in G} \left( \alpha(t) ( \alpha'(t^{-1}h)+ \beta'(gt)(t^{-1}h)) + \beta(g)(t) (\alpha'(t^{-1}h) +  \beta'(gt)(t^{-1}h)\right)x(gh) 
\\
& = \sum_{t\in G} ( \alpha(t) + \beta(g)(t))\sum_{ r \in G}\left( \alpha'(r) +  \beta'(gt)(r) \right) x(gtr) \qquad \qquad  (r=t^{-1}h) \\
& = \sum_{t \in G} \left( \alpha(t) + \beta(g)(t)\right) y(gt) \\
& = \tau^\omega(y)(g) \\
& = \left( \Psi(\omega)\circ \Psi(\omega')\right) (x)(g). 
\end{align*}
\par 
It follows that $\Psi(\omega * \omega')=\Psi(\omega)\circ \Psi(\omega')$. The proof is thus complete. 
\end{proof}

\par 
Observe that Theorem~\ref{t:iso-ring-first} does not hold when $G$ is a finite group since the ring morphism $\Psi$ fails to be injective in this case.

\section{Stable finiteness of generalized group rings and stably $L$-surjunctive groups}
\label{s:7}
 The main goal of the present section is to show that the stable finiteness property of the generalized group ring $D^1(k[G])$ is equivalent to the surjunctivity property of the classes $\mathrm{LUNCA}_c(G,k^n)$ for every $n \geq 1$ (Theorem~\ref{t:equivalence-L-1-stable}).  
\par 
\par 
We begin with the following  isomorphism between the ring $M_n(D^1(k[G]))$ of square matrices of size $n \times n$ with coefficients in the generalized group ring $D^1(k[G])$ and the ring $D^1(M_n(k[G]))$. 
\begin{proposition}
\label{p:iso-ring-second}
 Let $k$ be a ring and let $G$ be a group. Then for every  integer $n \geq 1$, there exists a canonical ring isomorphism: 
\begin{equation}
    D^1(M_n(k)[G])  \simeq  M_n(D^1(k[G])). 
\end{equation}
\end{proposition}

\begin{proof}
By \cite[Lemma~9.4]{phung-2020}, there exists a canonical  isomorphism  of rings 
$ M_n(k)[G] \simeq   M_n(k[G])$ given by $\sum_{g\in G}A(g) g \mapsto (\sum_{g \in G} A(g)_{ij}g)_{1 \leq i,j \leq n}$. 
Consider the map 
 $F\colon  D^1(M_n(k)[G])  \to  M_n(D^1(k[G]))$ defined as follows. For $ x= (\alpha, \beta) \in D^1(M_n(k)[G])$, we can write 
$\beta= \sum_{g \in G} \beta(g)g \in (M_n(k)[G])[G]$ where $\beta(g) = 
(\beta(g)_{ij})_{1 \leq i,j \leq n} \in M_n(k)[G]$ for all $g \in G$.  
Then we define $F(x) \in M_n(D^1(k[G]))$ by setting for all $1 \leq i,j \leq n$: 
\[
F(x)_{ij} = (\alpha_{ij}, \sum_{g \in G} \beta(g)_{ij} g) \in D^1(k[G]). 
\] 
\par 
It is clear that $F$ is a bijective homomorphism of groups and that 
$F((I_n, 0))= J_n$ where $I_n \in M_n(k)[G]$ and $J_n \in M_n(D^1(k[G]))$ are identity matrices of  $M_n(k)[G]$ and $M_n(D^1(k[G]))$ respectively. 
\par 
Now let $x_i=(\alpha_i, \beta_i)\in D^1(M_n(k)[G])$ for $i=1,2$. Then 
$x_1*x_2=(\alpha_1 \alpha_2, \alpha_1 \beta_2+ \beta_1 \alpha_2+\beta_1 \beta_2)$ and thus 
\begin{align}
\label{e:three-iso-canonical-1}
    F(x_1*x_2)_{ij} = ((\alpha_1\alpha_2)_{ij}, \sum_{g \in G} ((\alpha_1 \beta_2+ \beta_1 \alpha_2+\beta_1 \beta_2)(g))_{ij} g ).
\end{align}
\par 
On the other hand, we find that 
\begin{align*}
    (F(x_1)F(x_2))_{ij}   = \sum_{r=1}^n F(x_1)_{ir}*F(x_2)_{rj} = \sum_{r=1}^n ((\alpha_1)_{ir},   (\beta_1)_{ir})*((\alpha_2)_{rj},   (\beta_2)_{rj}). 
\end{align*}
\par 
 Therefore, if we denote $(F(x_1)F(x_2))_{ij}= (u,v)$ then 
 \begin{align}
 \label{e:three-iso-canonical-2}
 u= \sum_{r=1}^n (\alpha_1)_{ir} (\alpha_2)_{rj} = (\alpha_1\alpha_2)_{ij}
 \end{align}
 by the definition of matrix multiplication. Moreover, we deduce from the definition of the operation $*$ that $v= \sum_{g \in G} v(g) g$ where:  
 \begin{align*}
     v(g) & = \sum_{r=1}^n   ((\alpha_1)_{ir}(\beta_2)_{rj})(g) + 
    ((\beta_1)_{ir}(\alpha_2)_{rj})(g)+
    ((\beta_1)_{ir}(\beta_2)_{rj})(g).
 \end{align*}
 \par 
We infer from \eqref{e:def-singular-multiplication} that: 

\begin{align*}
    \sum_{r=1}^n ((\alpha_1)_{ir}(\beta_2)_{rj})(g)
    & =  \sum_{r=1}^n \sum_{h \in G}((\alpha_1)_{ir}(\beta_2)_{rj})(g)(h) h \\ 
    & =  \sum_{r=1}^n \sum_{h, t \in G} \alpha_1(t)_{ir} \beta_2(gt)(t^{-1}h)_{rj} h\\
    & =  \sum_{h\in G} \sum_{t\in G}\sum_{r=1}^n  \alpha_1(t)_{ir} \beta_2(gt)(t^{-1}h)_{rj} h\\
    & =  \sum_{h \in G} (\alpha_1 \beta_2)(g)(h)_{ij} h  \\
    & = (\alpha_1 \beta_2)(g)_{ij}. 
\end{align*}
\par 
Similarly, we have the equalities $ \sum_{r=1}^n ((\beta_1)_{ir}(\alpha_2)_{rj})(g) = (\beta_1 \alpha_2)(g)_{ij}$ and also 
$ \sum_{r=1}^n ((\beta_1)_{ir}(\beta_2)_{rj})(g) = (\beta_1 \beta_2)(g)_{ij}$. 
 Comparing to \eqref{e:three-iso-canonical-1}, it follows that $v$ is equal to the singular part of $F(x_1*x_2)_{ij}$. Consequently, we deduce from \eqref{e:three-iso-canonical-2} that for all $1 \leq i,j \leq n$, we have: 
 \[
(F(x_1)F(x_2))_{ij} = F(x_1*x_2)_{ij}. 
 \]
 \par 
 Hence, $F(x_1)F(x_2)= F(x_1*x_2)$ and we can finally  conclude that $F$ is a ring automorphism. The proof is thus complete. 
\end{proof}
\par 
The main result of the section is the following dynamical characterization of the direct finiteness of the ring $M_n(D^1(k[G]))$.  
\par 
\begin{theorem}
\label{t:equivalence-L-1-stable}
Let $G$ be an infinite group and let $k$ be a field. Then for every integer $n \geq 1$, the following are equivalent: 
    \begin{enumerate}[\rm (i)]
    \item 
    every  stably injective $\tau \in \mathrm{LNUCA}_c(G,k^{n})$   is surjective; 
     \item 
    the ring $\mathrm{LNUCA}_c(G,k^n)$ is directly finite; 
    \item  the ring $M_n(D^1(k[G]))$ is directly finite. 
    \end{enumerate}
\end{theorem}

\begin{proof}
    The equivalence between (ii) and (iii) results directly from Proposition~\ref{p:iso-ring-second} and Theorem~\ref{t:iso-ring-first} which imply that $\mathrm{LNUCA}_c(G,k^n) \simeq M_n(D^1(k[G]))$. \par 
    Suppose that (i) holds and let $\tau, \sigma \in \mathrm{LNUCA}_c(G,k^n)$ be two linear NUCA such that $\tau\circ \sigma= \Id$. Then Theorem~\ref{t:converse-general-stably-injective-left-invertible-linear} implies that $\sigma$ is stably injective. Consequently, we infer from (i) that $\sigma$ is surjective. In particular, $\sigma$ is bijective and thus so is $\tau$. It follows from $\tau \circ \sigma= \Id$ that $\sigma \circ \tau= \Id$ as well. This shows that the ring $\mathrm{LNUCA}_c(G,k^n)$ is directly finite. Therefore, we have shown that (i) implies (ii). 
    \par
    Suppose now that (ii) holds and let $\tau \in \mathrm{LNUCA}_c(G,k^n)$ be a stably injective linear NUCA. Then we deduce from Theorem~\ref{t:general-stably-injective-left-invertible-linear} that $\tau$ is left-invertible, i.e., there exists $\sigma \in \mathrm{LNUCA}_c(G,k^n)$ such that $\sigma \circ \tau = \Id$. Hence, (ii) implies that $\tau \circ \sigma = \Id$ and it follows at once that $\tau$ is surjective. Therefore, we also have that (ii) implies (i). The proof is thus complete. 
\end{proof}

\begin{corollary}
\label{c:equivalence-L-1-stable}
    Suppose that $G$ is an infinite group. Then the following are equivalent: 
    \begin{enumerate}[\rm (a)]
    \item 
    the group $G$ is $L^1$-surjunctive; 
    \item for every field $k$, the ring $D^1(k[G])$ is stably finite. 
    \end{enumerate}
\end{corollary}

\begin{proof}
    It is a direct consequence of Theorem~\ref{t:equivalence-L-1-stable}. 
\end{proof}

\section{Stable finiteness of generalized group rings and  $L^1$-surjunctive groups} 
\label{s:main-general} 

Extending \cite[Theorem~B]{phung-weakly}, we establish various characterizations of the stable finiteness of the ring $D^1(k[G])$ (for all field $k$)  notably in terms of the finitely $L^1$-surjunctivity of the group $G$. For ease of reading, we recall the statement of Theorem`\ref{t:main-first-intro} in the Introduction. 

\begin{theorem}
For every infinite group $G$, the following are equivalent:       
\begin{enumerate}[\rm (i)]
\item 
 $G$ is $L^1$-surjunctive; 
\item 
 $G$ is finitely $L^1$-surjunctive;
\item for every field $k$, the ring $D^1(k[G])$ is stably finite; 
\item 
for every finite field $k$, the ring $D^1(k[G])$ is stably finite; 
\item 
$G$ is dual $L^1$-surjunctive; 
\item 
$G$ is finitely dual $L^1$-surjunctive. 
\end{enumerate}
\end{theorem}

\begin{proof}
Let $V$ be a finite-dimensional vector space and let  $\tau \in \mathrm{LNUCA}_c(G,V)$. Then we obtain a dual linear NUCA $\tau^* \in \mathrm{LNUCA}_c(G,V)$ whose dual is exactly $\tau$, that is, $(\tau^*)^* = \tau$ (see \cite{phung-dual-linear-nuca}).   
We infer from  
\cite[Theorem~A]{phung-dual-linear-nuca}  that $\tau$ is pre-injective if and only if $\tau^*$ is surjective and that $\tau^*$ is stably injective if and only if $\tau$ is stably post-surjective. Hence, we deduce immediately the equivalences (i)$\iff$(v) and (ii)$\iff$(vi). 
\par 
The equivalence (i)$\iff$(ii) is the content of Corollary~\ref{c:equivalence-L-1-stable}. Similarly, the exact same proof of Theorem~\ref{t:equivalence-L-1-stable} shows that (ii)$\iff$(iv). Finally, the equivalence (i)$\iff$(ii) results from 
Theorem~\ref{t:main-reduction-grothendieck} below. 
The proof is thus complete. 
\end{proof}
\par 
Our next key result extends \cite[Theorem~A]{phung-weakly}. The proof follows quite closely the reduction strategy of the proof of \cite[Theorem~A]{phung-weakly} and 
\cite[Theorem~B]{phung-geometric} which is   less involved in our linear case. 

\begin{theorem}
    \label{t:main-reduction-grothendieck}  Let $G$ be a   group and let $n\geq 1$ be an integer. Then the following are equivalent:      
    \begin{enumerate}[\rm (i)]
\item 
for every finite field $k$, all stably injective $\tau \in \mathrm{LNUCA}_c(G,k^n)$ are surjective; 
\item 
for every field $k$, all stably injective $\tau \in \mathrm{LNUCA}_c(G,k^n)$ are surjective. 
\end{enumerate}
\end{theorem}

\begin{proof} 
Since the case when $G$ is finite is clear and (ii)$\implies$(i) trivially,  we suppose in the rest of the proof that $G$ is an infinite group which satisfies (i).  
Let $V$ be a finite-dimensional vector space over a field $k$ (not necessarily finite) and let $\tau \in \mathrm{LNUCA}_c(G, V)$. Suppose that $\tau$ is stably injective. Then by definition we can choose a finite subset $M \subset G$ with $1_G \in M=M^{-1}$ and a configuration $s \in \LL(V^M,V)^G$ which is asymptotic to a constant configuration $c \in \LL(V^M,V)^G$ such that $\tau=\sigma_s$ and $s_{G \setminus M}=c\vert_{G \setminus M}$. We infer from Theorem~\ref{t:general-stably-injective-left-invertible-linear} that $\tau$ is left-invertible. Hence, we can find $\sigma \in \mathrm{LNUCA}_c(G,V)$ such that $\sigma \circ \tau =\Id$. Moreover, up to enlarging the finite set $M$, we can find $t \in \LL(V^M, V)^G$ asymptotic to a constant configuration $d\in \LL(V^M, V)^G$ such that $\sigma = \sigma_t$ and $t\vert_{G \setminus M}=d\vert_{G \setminus M}$. 
\par 
Let us denote $\Gamma= \tau(V^G)$. 
As $\sigma_t \circ \sigma_s= \Id$, we deduce for all $g \in G$ that 
\begin{equation}
    \label{e:local-magic-main-reduction}
    f^+_{\{g\}, t(g)}  \circ f^+_{gM, s \vert_{gM}} = \pi_{gM^2, \{g\}}
\end{equation}
where $\pi_{F, E} \colon V^F\to V^E$ denotes the canonical projection for all sets $E \subset F$.  Consider the similar condition where we switch the role of $s$ and $t$:  
\begin{equation}
    \label{e:local-magic-main-reduction-converse}
    f^+_{\{g\}, s(g)}  \circ f^+_{gM, t \vert_{gM}} = \pi_{gM^2, \{g\}}. 
\end{equation}
\par 
Since $G$ is infinite, we can choose a finite subset $M^*\subset G$ such that $M^2 \subsetneq M^*$. Then observe that $\sigma_t \circ \sigma_s= \Id$, resp. $\sigma_s \circ \sigma_t= \Id$, if and only if \eqref{e:local-magic-main-reduction}, resp.  
\eqref{e:local-magic-main-reduction-converse},  holds for all $g \in M^*$ (see \cite[Lemma~2.2]{phung-lef} for the case of CA). Hence, up to  making the base change to $k'$ (replacing $V$, $s(g)$, $t(g)$ resp. by $V \otimes_k k'$, $s(g) \otimes_k k'$, $t(g) \otimes_k k'$ etc.) where $k'$ is an  algebraically closed field which contains $k$, we can suppose without loss of generality that $k$ is algebraically closed. 
\par 
We obtain from  \cite[Lemma~2.1]{phung-weakly} a finitely generated $\Z$-algebra $R \subset k$ and $R$-modules  of finite type $V_R$ and: 
\[
s_R, t_R \in \mathrm{Hom}_{R-mod}((V_R)^M, V_R)^G 
\] 
such that for some fixed $g_0 \in M^*\setminus M^2$, $s_R(g) = s_R(g_0)$, $t_R(g)=t_R(g_0)$ for all $g \in G \setminus M^*$ and 
the following hold: 
\begin{enumerate}[\rm I.] 
    \item 
    $V= V_R \otimes_R k$,  
    \item   $s(g)= s_R(g) \otimes_R k$ and  $t(g)= t_R(g) \otimes_R k$ for all $g \in M^*$. 
    \end{enumerate}
    where $\pi_R \colon (V_R)^{E_n M} \to (V_R)^{\{1_G\}}$ is the canonical projection. Essentially, we can take $R=\Z[\Omega]$ where $\Omega \subset k$ is a finite subset consisting of the entries of the matrices which represent  the linear maps $s(g)$, $t(g)$ for all $g \in M^*$. 
\par 
Let us denote $S=\Spec R$ which is a $\Z$-scheme of finite type. 
Then we infer from \cite[Lemma~2.2]{phung-weakly} that the set of closed points of $V_R^{E_n}$ is given by $\Delta = \cup_{p \in \PP, a \in S_p,  d\in \N} H_{p,a,d}^{E_n}$. Here,  
 $\PP$ denotes the set of prime numbers. By $a \in S_p= S \otimes_\Z \mathbb{F}_p$ we mean $a$ is a closed point of $S_p$. In particular, $\kappa(a)$ is a finite field. The set $H_{p,a, d}$ is defined by: 
\begin{equation}
   H_{p,a, d}= \{x \in V_a \colon \vert \kappa(x) \vert=p^r, 1 \leq r \leq d\} 
\end{equation}
which is a finite linear subspace of the finite-dimensional $\kappa(a)$-vector space $V_a=V_R \otimes_R \kappa(a)$.  
\par 
Let us fix $p \in \PP$, $a \in S_p$, $d \in \N$ and 
consider the configurations of local defining maps  $s_a, t_a \in \LL(V_a^M, V_a)^G$ where for all element $g \in G$, we define $s_a(g) = s_R(g)\otimes_R \kappa(a)$ 
and $t_a(g) = t_R(g)\otimes_R \kappa(a)$. 
Observe that $s_{a}(g)(H_{p,a,d}^M)$ and $t_{a}(g)(H_{p,a,d}^M)$ are subsets of $H_{p,a,d}$ for all $g \in G$ (cf. e.g. \cite[Lemma~3.1]{phung-geometric}). Consequently, we can define $s_{p,a,d}, t_{p,a,d} \in \LL(H_{p,a,d}^M, H_{p,a,d})$ by setting $s_{p,a,d} = s_a\vert_{H_{p,a,d}}$ and $t_{p,a,d} = t_a\vert_{H_{p,a,d}}$ for all $g \in G$. Thus, we obtain well-defined linear NUCA 
$\sigma_{s_{p,a,d}}, \sigma_{t_{p,a,d}} \colon H_{p,a,d}^G \to H_{p,a,d}^G$. 
\par 
From \eqref{e:local-magic-main-reduction}, it is clear from our construction that for all $g \in M^*$, we have: 
\begin{equation}
    \label{e:reduction-local-magic-main-reduction}
    f^+_{\{g\}, t_{p,a,d}(g)}  \circ f^+_{gM, s_{p,a,d} \vert_{gM}} = \pi^{p,a,d}_{gM^2, \{g\}}
\end{equation}
where $\pi^{p,a,d}_{F, E} \colon H_{p,a,d}^F\to H_{p,a,d}^E$ denotes the canonical projection for all sets $E \subset F$. It follows that $\sigma_{t_{p,a,d}}\circ \sigma_{s_{p,a,d}}=\Id$. In particular, $\sigma_{s_{p,a,d}}$ is left-invertible and we deduce from Theorem~\ref{t:converse-general-stably-injective-left-invertible-linear} that $\sigma_{s_{p,a,d}}$ is stably  injective. Since (i) holds by hypothesis and $H_{p,a,d}$ is a finite $\kappa(a)$-vector space, $\sigma_{s_{p,a,d}}$ is surjective. It follows at once that $\sigma_{s_{p,a,d}}\circ \sigma_{t_{p,a,d}}=\Id$. 
\par 
Therefore, we deduce that for every $g \in M^*$, the equality:   
\begin{equation}
    \label{e:reduction-local-magic-main-reduction-converse}
    f^+_{\{g\}, s_R(g)}  \circ f^+_{gM, t_R \vert_{gM}} = \pi^R_{gM^2, \{g\}} 
\end{equation} 
where $\pi^R_{gM^2, \{g\}} \colon V_R^{gM^2} \to V_R^{\{g\}}$, 
holds over the set $\Delta = \cup_{p \in \PP, s \in S_p,  d\in \N} H_{p,s,d}^{gM^2}$ of all closed points of $(V_R)^{gM^2}$. 
Since $V_R$ is a Jacobson scheme (cf., e.g. \cite[Section~3]{phung-geometric}), an argument using the equalizer as in \cite[Lemma~7.2]{cscp-alg-ca} shows that 
$ f^+_{\{g\}, s_R(g)}  \circ f^+_{gM, t_R \vert_{gM}} = \pi^R_{gM^2, \{g\}} $ as $R$-morphisms $V_R^{gM^2} \to V_R^{\{g\}}$. 
Consequently, we obtain the relation \eqref{e:local-magic-main-reduction-converse} for all $g \in M^*$ 
by making the base change \eqref{e:reduction-local-magic-main-reduction-converse}$\otimes_R k$. 
It follows that $\sigma_s \circ \sigma_t=\Id$ and we can finally conclude that $\tau=\sigma_s$ is surjective. Therefore, we also have (i)$\implies$(ii) and the proof is complete. 
\end{proof}

\section{Applications} 
\label{s:9}

For the proof of Theorem~\ref{t:intro-initially-subamenable}, we first establish the following extension of \cite[Theorem~B]{phung-tcs} and \cite[Theorem~D]{phung-dual-linear-nuca} to cover the case of initially subamenable group universes and finite vector space alphabets:  

\begin{theorem}
\label{t:application-finitely-l-1} 
Every initially amenable group is finitely $L^1$-surjunctive. 
\end{theorem}

\begin{proof}
Let $G$ be an initially subamenable group and let $V$ be a finite vector space. Suppose that $\tau\in \mathrm{LNUCA}_c(G,V)$ is stably injective. Then we can infer without difficulty from \cite[Theorem~A]{phung-tcs} or \cite[Theorem~B]{phung-dual-linear-nuca} that there exist a large enough finite subset $M \subset G$ and two configurations $s,t \in \LL(V^M, V)^G$ and another  configuration $c \in \LL(V^M, V)^G$ such that $\tau= \sigma_s$, $s\vert_{G \setminus M} = t\vert_{G \setminus M} = 
c\vert_{G \setminus M}$, and $\sigma_t \circ \sigma_s= \Id$. Up to enlarging $M$, we can  suppose without loss of generality that $1_G \in M$ and $M=M^{-1}$. 
\par 
If $G=M^4$ then $G$ is finite and the theorem is trivial since every injective endomorphism of $V^G$ is surjective. Consider the case $M^4 \subsetneq G$. 
Let $E \subset G$ be any finite subset which contains strictly $M^4$, that is, $M^4 \subsetneq E$.   
Since the group $G$ is initially subamenable, we can find an amenable group $H$ and an injective map 
$\varphi \colon E \to H$ such that $\varphi(gh)=\varphi(g) \varphi(h)$ for all $g,h \in E$ such that $gh \in E$. In particular, $\varphi(gh)=\varphi(g) \varphi(h)$ for all $g,h \in M$. Since $M=M^{-1}$ and $1_G \in M$, we deduce that $1_H \in \varphi(M)$ and $\varphi(M)=\varphi(M)^{-1}$. 
\par 
Up to replacing $H$ by the subgroup generated by $\varphi(E)$, we can suppose that $H$ is generated by $\varphi(E)$. As $\sigma_t \circ \sigma_s= \Id$, we deduce for all $g \in E$ that 
\begin{equation}
    \label{e:application-1}
    f^+_{\{g\}, t(g)}  \circ f^+_{gM, s \vert_{gM}} = \pi_{gM^2, \{g\}}
\end{equation}
where we denote by $\pi_{F, Q} \colon V^F\to V^Q$  the canonical linear projection for all sets $Q \subset F$. 
The bijection $\varphi\vert_E \colon E \to  \varphi(E)$ induces in particular an isomorphism $\phi \colon  V^{\varphi(M)} \to V^M$. Let us fix $g_0 \in E \setminus M$. 
The patterns $s\vert_{E}$, $t\vert_E$ in turn induce the configurations $\tilde{s}, \tilde{t}\in \LL(V^{\varphi(M)}, V)^H$ defined by 
$\tilde{s}(h) = s(h) \circ \phi$, $\tilde{t}(h) = t(h) \circ \phi$ for all $h \in \varphi(E)$ and 
$\tilde{s}(h) = s(g_0) \circ \phi$, $\tilde{t}(h) = t(g_0) \circ \phi$ for all $h \in H \setminus \varphi(E)$. 
\par 
Since $\varphi$ is injective, it follows from \eqref{e:application-1} that for all $h \in \varphi(E)$, we have:  
\begin{equation}
    \label{e:application-2}
    f^+_{\{h\}, \tilde{t}(h)}  \circ f^+_{h\varphi(M), \tilde{s} \vert_{h\varphi(M)}} = \pi_{h\varphi(M^2), \{h\}}. 
\end{equation}
\par 
Consequently, we deduce that $\sigma_{\tilde{t}}\circ \sigma_{\tilde{s}}=\Id$. In particular, $\sigma_{\tilde{s}}$ is injective. Since $H$ is amenable and $\tilde{s}$ is asymptotically constant by construction, we infer from \cite[Theorem~B.(i)]{phung-tcs} that $\tilde{s}$ is surjective. Hence, it follows from $\sigma_{\tilde{t}}\circ \sigma_{\tilde{s}}=\Id$ that 
$\sigma_{\tilde{s}}\circ \sigma_{\tilde{t}}=\Id$. We deduce that for every $h \in \varphi(E)$, we have: 
\begin{equation}
    \label{e:application-3}
    f^+_{\{h\}, \tilde{s}(h)}  \circ f^+_{h\varphi(M), \tilde{t} \vert_{h\varphi(M)}} = \pi_{h\varphi(M^2), \{h\}}. 
\end{equation}
\par 
Therefore, via the injection $\varphi$, we obtain for all $g \in E$ that 
\begin{equation}
    \label{e:application-4}
    f^+_{\{g\}, s(g)}  \circ f^+_{gM, t \vert_{gM}} = \pi_{gM^2, \{g\}}.
\end{equation}
\par 
By the choice of $E$ and $\varphi$, we can thus conclude that $\sigma_s \circ \sigma_t= \Id$ which implies in particular that $\sigma_s$ is surjective. The proof is thus complete. 
\end{proof}

\par 
Observe that by a similar  argument, \cite[Theorem~B.(i)]{phung-tcs} also holds for initially subamenable group universes. 
As an immediate consequence of Theorem~\ref{t:main-first-intro} and Theorem~\ref{t:application-finitely-l-1}, we obtain the proof of Theorem~\ref{t:intro-initially-subamenable} in the Introduction as follows:  
\begin{proof}[Proof of Theorem~\ref{t:intro-initially-subamenable}] 
Thanks to Theorem~\ref{t:main-first-intro}, 
we infer respectively from from Theorem~\ref{t:application-finitely-l-1} and \cite[Theorem~B.(ii)]{phung-tcs} that 
all initially amenable groups and all residually finite groups are $L^1$-surjunctive. We can thus conclude the proof of the theorem since dual $L^1$-surjunctivity is equivalent to $L^1$-surjunctivity 
also by Theorem~\ref{t:main-first-intro}. 
\end{proof}

\bibliographystyle{siam}

\end{document}